\documentclass[11pt]{article}

\usepackage{booktabs,amsmath}
\usepackage{amssymb}
\usepackage{amsthm}
\usepackage{graphics}
\usepackage{tikz}
\usepackage{tikz-qtree}
\usepackage{marvosym}
\usetikzlibrary{decorations.pathreplacing}
\usepackage{amsfonts}
\usepackage{mathtools}
\usepackage{algorithm}
\usepackage{mathdots}
\usepackage[noend]{algpseudocode}
\usepackage{float}
\usepackage{caption}
\usepackage{subfigure}
\allowdisplaybreaks

\newtheorem{lemma}{Lemma}
\newtheorem{theorem}{Theorem}
\newtheorem{proposition}{Proposition}
\newtheorem{corollary}{Corollary}

\title{Mean Row Values in $(u,v)$-Calkin-Wilf Trees}
\author{Sandie Han, Ariane M. Masuda, Satyanand Singh, and Johann Thiel}

\begin{document}

\maketitle

\abstract{We fix integers $u,v \geq 1$, and consider an infinite binary tree $\mathcal{T}^{(u,v)}(z)$ with a root node whose value is a positive rational number $z$. For every vertex $a/b$, we label the left child as $a/(ua+b)$ and right child as $(a+vb)/b$.  The resulting tree is known as the $(u,v)$-Calkin-Wilf tree. As $z$ runs over  $[1/u,v]\cap \mathbb{Q}$, the vertex sets of 
$\mathcal{T}^{(u,v)}(z)$ form a partition of $\mathbb{Q}^+$. When $u=v=1$, the mean row value converges to $3/2$ as the row depth increases.  Our goal is to extend this result for any $u,v\geq 1$. We show that, when $z\in [1/u,v]\cap \mathbb{Q}$, the mean row value in  $\mathcal{T}^{(u,v)}(z)$ converges to a value close to $v+\log 2/u$ uniformly on $z$. }

\section{Introduction}
\label{sec:1}
In~\cite{N}, Nathanson defines an infinite binary tree generated by the following rules:
\begin{enumerate}
    \item fix two positive integers $u$ and $v$,
    \item label the root of the tree by a rational $z$, and
    \item for any vertex labeled $\dfrac{a}{b}$, label its left and right children by $\dfrac{a}{ua+b}$ and $\dfrac{a+vb}{b}$, respectively.
\end{enumerate}
In the case where $u$, $v$, and $z$ are equal to 1, the tree generated is the well-known Calkin-Wilf tree~\cite{CW} (see Figure~\ref{fig:CWtree}). Since Nathanson's definition represents a generalization\footnote{For other generalizations, see~\cite{BM,MS}.} of the Calkin-Wilf tree, we refer to trees defined in the above manner as $(u,v)$-Calkin-Wilf trees, and we denote them by $\mathcal{T}^{(u,v)}(z)$ (see Figure~\ref{fig:uvgraph}). The set of  depth $n$ vertices of $\mathcal{T}^{(u,v)}(z)$ is denoted by $\mathcal{T}^{(u,v)}(z;n)$. For example, we see from Figure~\ref{fig:CWtree} that $\mathcal{T}^{(1,1)}(1;1)=\{1/2,2\}.$

\begin{figure}[ht!]
\begin{center}
\begin{tikzpicture}[sibling distance=15pt]
\tikzset{level distance=30pt}
\Tree[.$1/1$ [.$1/2$ [.$1/3$ $1/4$ $4/3$ ]
   [.$3/2$ $3/5$ $5/2$ ] ] [.$2/1$ [.$2/3$  $2/5$ $5/3$ ] [.$3/1$ $3/4$ $4/1$ ] ]]
\end{tikzpicture}
\caption{The first four rows of the Calkin-Wilf tree.}\label{fig:CWtree}
\end{center}
\end{figure}
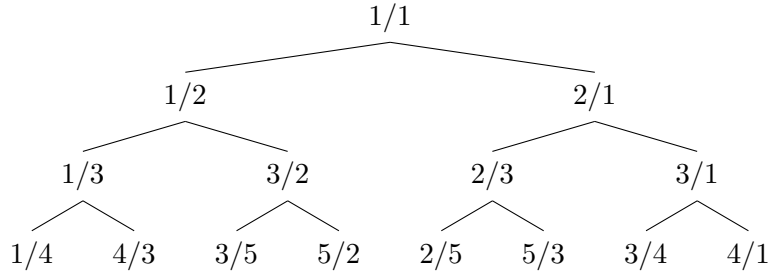

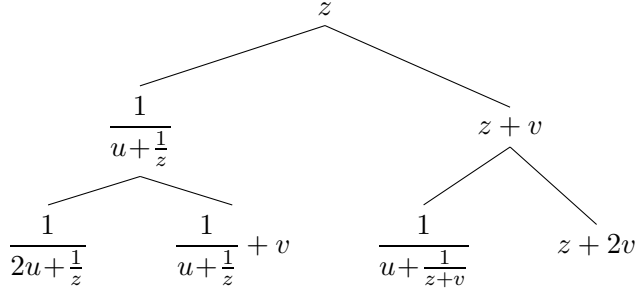
\begin{figure}[ht!]
\begin{center}
\begin{tikzpicture}[sibling distance=25pt]
\tikzset{level distance=45pt}
\Tree[.$z$ [.$\dfrac{1}{u\!+\!\frac{1}{z}}$ [.$\dfrac{1}{2u\!+\!\frac{1}{z}}$ ]
[.$\dfrac{1}{u\!+\!\frac{1}{z}}+v$ ] ] [.$z+v$  [.$\dfrac{1}{u\!+\!\frac{1}{z+v}}$ ] [.$z+2v$ ] ] ]
\end{tikzpicture}
\caption{The first three rows of $\mathcal T^{\;(u,v)}(z)$.}\label{fig:uvgraph}
\end{center}
\end{figure}

The vertices of $\mathcal{T}^{(1,1)}(1)$ are all positive rational numbers without any repetition~\cite{CW}.  More generally, the trees $\mathcal{T}^{(u,v)}(z)$ form a partition of
$\mathbb{Q}^+$ as $z$ runs over  $[1/u,v]\cap \mathbb{Q}$; see~\cite{N}.    The Calkin-Wilf tree has many other interesting properties~\cite{CW, HMST1,HMST2, N, Ne}, one of which is the fact that the mean value of vertices of depth $n$ converges to $3/2$ as $n\to\infty$~\cite{A,R}. Our main result generalizes this property for all $(u,v)$-Calkin-Wilf trees.

The proof that the mean value of vertices of depth $n$ converges to $3/2$ is not difficult and only makes use of one  property of the Calkin-Wilf tree; namely, both $a/b$ and $b/a$ appear (in symmetric positions) on every row; see Figure~\ref{fig:CWtree}.

\begin{proposition}\label{symprop}
If $\dfrac{a}{b}\in\mathcal{T}^{(1,1)}(1;n)$, then $\dfrac{b}{a}\in\mathcal{T}^{(1,1)}(1;n)$. 
\end{proposition}

 The proof of Proposition~\ref{symprop} follows quickly from induction on the depth $n$. We omit the details.

\begin{theorem}\label{11mean}
For $n\geq 0$, let $A(n) = \dfrac{1}{2^n}\displaystyle\sum_{y\in\mathcal{T}^{(1,1)}(1;n)}y$. Then $\displaystyle\lim_{n\to\infty}A(n)=\dfrac{3}{2}.$
\end{theorem}

\begin{proof}
Let $S(n)=\displaystyle\sum_{y\in\mathcal{T}^{(1,1)}(1;n)}y$. Rewriting $y$ as $a/b$ and using both the definition of the Calkin-Wilf tree and Proposition~\ref{symprop}, we see that, for $n\geq 1$, 
\begin{align*}
    2S(n) &= \sum_{\frac{a}{b}\in\mathcal{T}^{(1,1)}(1;n-1)}\left(\frac{a}{a+b} + \frac{a}{b} + 1 + \frac{b}{b+a} + \frac{b}{a} + 1\right)\\
    &= \sum_{\frac{a}{b}\in\mathcal{T}^{(1,1)}(1;n-1)}\left(\frac{a}{b} + \frac{b}{a} + 3\right)\\
    &= 2S(n-1) + 3\cdot 2^{n-1}.
\end{align*}
This gives the recurrence relation $S(0)=1$ and $S(n)=S(n-1)+3\cdot 2^{n-2}$ for $n\geq 1$. Solving the recurrence relation gives that $S(n)=\frac{3}{2}\cdot 2^n-\frac{1}{2}$ for $n\geq 0$. The desired result follows immediately since $A(n)=S(n)/2^n.$  
\end{proof}

Let $S^{(u,v)}(z;n)=\displaystyle\sum_{y\in\mathcal{T}^{(u,v)}(z;n)}y$ and $A^{(u,v)}(z;n)=S^{(u,v)}(z;n)/2^n$. Suppose $uv>1$. As a consequence of Lemma~\ref{mono} and Theorem~\ref{mainthm}, we show  that if $z\in[1/u,v]\cap\mathbb{Q}$, then $\displaystyle\lim_{n\to\infty}A^{(u,v)}(z;n)$ exists\footnote{The reason for limiting our choice of roots to $[1/u,v]\cap\mathbb{Q}$ is that these rationals are the ``orphan" roots in the sense that they are not the children of {\em any} rational in {\em any} $(u,v)$-Calkin-Wilf tree~\cite{N}.}, that the limit is independent of the value of $z$, and that the limit has a value close to $v+\log{2}/u.$ Unfortunately, Proposition~\ref{symprop} does not generalize to other $(u,v)$-Calkin-Wilf trees by Lemma~\ref{cfterms}, so a different approach is needed in this broader setting.

At first the value $v+\log{2}/u$ may seem surprising, but a simple heuristic argument quickly leads to this quantity. Note that if $a/b$ is a vertex in a $(u,v)$-Calkin-Wilf tree, then its children are given by $$\frac{a}{ua+b}=\frac{1}{u+\frac{b}{a}}<\dfrac{1}{u}\quad\text{ and }\quad \frac{a+vb}{b} = \frac{a}{b} + v >v.$$ Following this pattern from depth $n$ to depth $n+1$ suggests that a quarter of all elements of a fixed (large) depth have integer part of roughly size $v$, an eighth have integer part of roughly size $2v$, etc. Similarly, half of all elements have a fractional part of roughly size $1/u$, a quarter have a fractional part of roughly size $1/(2u)$, etc. So we expect that 
\begin{align*}
    A^{(u,v)}(z;n) &\approx \frac{1}{2^n}\left(\frac{2^n}{4}\left(v+\frac{2}{u}\right) + \frac{2^n}{8}\left(2v+\frac{2}{2u}\right) + \frac{2^n}{16}\left(3v+\frac{2}{3u}\right)+\cdots  \right)\\
    &= \frac{v}{4}\sum_{k=0}^\infty \frac{k+1}{2^k} + \frac{1}{u}\sum_{k=1}^\infty \frac{1}{k2^k}\\
    &= v + \frac{\log{2}}{u},
\end{align*}
where the last equality follows from the Taylor series expansions for $1/(1-x)^2$ and $\log(1-x)$.

This heuristic throws away a lot of information from the denominator in the fractional part of each element. We would therefore expect the true value of $A^{(u,v)}(z;n)$ to be smaller than $v+\log{2}/u.$ 

As for the independence of the limit of $A^{(u,v)}(z;n)$ from $z\in[1/u,v]\cap\mathbb{Q}$, we note that if $a/b$ is a vertex in a $(u,v)$-Calkin-Wilf tree with continued fraction representation $a/b=[q_0,q_1,\dots,q_r]$, then the children of $a/b$ have easily computable continued fractions, as the next result shows.
\begin{lemma}({\cite[Lemma 5]{HMST1}})\label{cf}
 Let $a/b$ be a positive rational number with continued fraction representation $a/b=[q_0,q_1,\dots,q_r]$. It follows that
 \begin{enumerate}
    \item[(a)] if $q_0=0$, then $a/(ua+b)=[0,u+q_1,\dots,q_r]$;
    \item[(b)] if $q_0\neq0$, then $a/(ua+b)=[0,u,q_0,q_1,\dots,q_r]$;
    \item[(c)] and $(a+vb)/b=[v+q_0,q_1,\dots,q_r]$.
 \end{enumerate}
\end{lemma}
It follows from the result above that, for large $n$, most vertices of depth $n$ will have approximately $n/2$ coefficients in their continued fraction expansions. This lowers the influence of the root on the value of $A^{(u,v)}(z;n)$ as it is quickly buried by the above process. We will make this notion precise in Lemma~\ref{mcount}.

\section{Main Result}
\label{sec:2}

We show that for $z\in[1/u,v]\cap\mathbb{Q}$, the limit of $A^{(u,v)}(z;n)$ exists as $n\to\infty$ in two main steps:
\begin{enumerate}
    \item[(A)] First we show that, for $z=1/u$ or $z=v$, the mean $A^{(u,v)}(z;n)$ is monotonic increasing and bounded above as $n\to\infty$.
    \item[(B)] Second we show that $A^{(u,v)}(z_1;n)-A^{(u,v)}(z_2;n)\to 0$ as $n\to\infty$ for any $z_1,z_2\in[1/u,v]\cap\mathbb{Q}$.
\end{enumerate}

We begin with a useful lemma for comparing rational numbers based on their continued fraction coefficients.

\begin{lemma}(\cite[p.~101]{Ro})\label{contineq}
	Suppose that $\alpha,\beta\in\mathbb{Q}$ are distinct with  $\alpha=[p_0,p_1,\dots,p_s]$ and $\beta=[q_0,q_1,\dots,q_r]$. Let $k$ be the smallest index such that $p_k\neq q_k$. Then $\alpha < \beta$ if and only if $p_k<q_k$ when $k$ is even and $p_k>q_k$ when $k$ is odd. If no such $k$ exists and $n<m$, then $\alpha < \beta$ if and only if $n$ is even.
\end{lemma}

We note here two useful results from~\cite{HMST1} that will be used to obtain our main result. Lemma~\ref{cfterms} and Corollary~\ref{cfdepth} show two things: that there is a very close relationship between two vertices in the same $(u,v)$-Calkin-Wilf tree via their continued fraction representations if one is the descendant of the other, and that the continued fraction representation of a vertex in a $(u,v)$-Calkin-Wilf tree encodes its depth in the tree.

\begin{lemma}({\cite[Theorem 3]{HMST1}})\label{cfterms}
Suppose that $z$ and $z'$ are positive rational numbers with continued fraction representations $z=[q_0,q_1,\dots,q_r]$ and $z'=[p_0,p_1,\dots,p_s]$. Then $z'$ is a descendant of $z$ in the $(u,v)$-Calkin-Wilf tree with root $z$ if and only if the following conditions all hold:
\begin{enumerate}
\item[(a)] $s\geq r$ and $2\mid (s-r)$;
\item[(b)]for $0\leq j\leq s-r-1$, $v\mid p_j$ when $j$ is even and $u\mid p_j$ when $j$ is odd;
\item[(c)] for $2\leq i\leq r$, $p_{s-r+i}=q_i$;
\item[(d)] and
\begin{enumerate}
\item[(i)] if $q_0\neq 0$, then $p_{s-r}\geq q_0$, $v\mid(p_{s-r}- q_0)$ and $p_{s-r+1}=q_1$;
\item[(ii)] otherwise, if $q_0 = 0$, then $v\mid p_{s-r}$, $p_{s-r+1}\geq q_1$, and $u\mid(p_{s-r+1}- q_1)$.
\end{enumerate}
\end{enumerate}
\end{lemma}

\begin{lemma}({\cite[Corollary 3]{HMST1}})\label{cfdepth}
Using the same hypothesis as Lemma~\ref{cfterms}, if $n$ is the depth of $z'$, then
\begin{align*}
n & = \frac{1}{v}\Bigg(\sum_{\substack{0\leq j\leq s-r-1\\j\text{ even}}}p_j + \sum_{\substack{0\leq i\leq r\\i \text{ even}}}(p_{s-r+i} - q_i)\Bigg) \\
&\qquad + \frac{1}{u}\Bigg(\sum_{\substack{0\leq j\leq s-r-1\\j \text{ odd}}}p_j + \sum_{\substack{0\leq i\leq r\\i \text{ odd}}}(p_{s-r+i} - q_i)\Bigg).
\end{align*}
\end{lemma}

The following lemma gives us the desired monotonicity for $A^{(u,v)}(z;n)$ when $z=1/u$ or $z=v$.

\begin{lemma}\label{mono}
	For any $n\geq 0$, if $z=1/u$ or $z=v$, then $S^{(u,v)}(z;n+1)>2S^{(u,v)}(z;n).$
\end{lemma}

\begin{proof}
	Let $n\geq 0$ be given. Enumerate the elements in $\mathcal{T}^{(u,v)}(z;n)$ and $\mathcal{T}^{(u,v)}(z;n+1)$ as they appear from left to right in the $(u,v)$-Calkin-Wilf tree by $s_0,s_1,\dots,s_{2^n-1}$ and $t_0,t_1,\dots,t_{2^{n+1}-1}$, respectively. Clearly, for $0\leq i\leq 2^n-1$, $t_{2i}$ and $t_{2i+1}$ are the left and right children of $s_i$. Our goal is therefore to show that $$2\sum_{i=0}^{2^n-1} s_i < \sum_{i=0}^{2^{n+1}-1} t_i.$$ This desired inequality can be reduced further by noting that $t_{2i+1} = s_i + v$. In other words, we obtain the desired result if we can show that $$\sum_{i=0}^{2^n-1}s_i<2^nv + \sum_{i=0}^{2^n-1}t_{2i}.$$ Let $\mathcal{I}_n =\sum_{i=0}^{2^n-1}[s_i]$. That is, $\mathcal{I}_n$ is the sum of the integer parts of all of the depth $n$ elements of the $(u,v)$-Calkin-Wilf tree.

	\underline{Claim:} $\mathcal{I}_n = (2^n-1)v+[w]$ for $n\geq 0$.

	We prove the above claim by induction. Clearly $\mathcal{I}_0 = [w]$. Suppose that the claim holds for some $k\geq 1$. Since the left child of any number appearing in the $(u,v)$-Calkin-Wilf tree is smaller than $1/u$ and the right child of any element is always the original element plus $v$, it follows that $\mathcal{I}_{k+1} = \mathcal{I}_k+2^kv$. By assumption, $\mathcal{I}_k=(2^k-1)v+[w]$, from which the desired result immediately follows.
	
	Our previous claim shows that we obtain the desired result if we can show that 
	\begin{align}
	[w]+\sum_{i=0}^{2^n-1}\{s_i\} &< v+\sum_{i=0}^{2^n-1}t_{2i}.\label{mainineq}
    \end{align}	
	
	If we take $w=1/u$, then $[w]=0$ and, by Lemma~\ref{cfdepth}, the short continued fraction representation of $\{s_i\}$ must be of the form $[0,\alpha_1 u,\alpha_2 v,\dots,\alpha_k u]$ with $m:=m(s_i)=n+2-\sum_{i=1}^k \alpha_i >0 .$ Since $\{s_{2^n-1}\}=[0,u]$ and $t_0=[0,(n+2)u]$, we see that, in this case,~\eqref{mainineq} reduces further to the inequality
	\begin{align}
	\sum_{i=0}^{2^n-2}\{s_i\} &< \sum_{i=1}^{2^n-1}t_{2i}.\label{reducedmainineq}
    \end{align}
    If $\alpha_k=1$, then there is an $1\leq i^*\leq 2^n-1$ such that $$t_{2i^*}=[0,\alpha_1 u,\alpha_2 v,\dots,(\alpha_{k-1}+1)v, mu].$$ If $\alpha_k> 1$, then there is an $1\leq i^*\leq 2^n-1$ such that $$t_{2i^*}=[0,\alpha_1 u,\alpha_2 v,\dots,(\alpha_k-1)u,v,mu].$$ In either case, it follows that $\{s_i\}<t_{2i^*}$ by Lemma~\ref{contineq}. Note that the above association between $\left\{\{s_i\}\right\}_{i=0}^{2^n-2}$ and $\{t_{2i}\}_{i=1}^{2^n-1}$ is bijective, from which~\eqref{mainineq} follows in this case.
    
    If we take $w=v$, then $[w]=v$ and, by Lemma~\ref{cfdepth}, the short continued fraction representation of $\{s_i\}$ must be of the form $[0,\alpha_1 u,\alpha_2 v,\dots,\alpha_k v]$ with $m$ defined as in the previous case. Since $\{s_{2^n-1}\}=0$ and $t_0=[0,(n+1)u,v]$, we see that, in this case,~\eqref{mainineq} also reduces to~\eqref{reducedmainineq}. If $m=1$, then there is an $1\leq i^*\leq 2^n-1$ such that $$t_{2i^*}=[0,\alpha_1 u,\alpha_2 v,\dots,(\alpha_k+1)v].$$ If $m>1$, then there is an $1\leq i^*\leq 2^n-1$ such that $$t_{2i^*}=[0,\alpha_1 u,\alpha_2 v,\dots,\alpha_kv,(m-1)u,v].$$ As in the previous case,~\eqref{mainineq} follows, completing the proof of the lemma.
     
\end{proof}

The following theorem establishes $v+\log{2}/u$ as an  upper bound of $A^{(u,v)}(z;n)$. Note that by $f(x)=O(g(x))$ we mean that $|f(x)|\leq C|g(x)|$ for some constant $C$ (which may differ depending on context) and all sufficiently large $x$.

\begin{theorem}\label{mainthm}
	If $u$ and $v$ are positive integers with $uv>1$ and $z\in\mathbb{Q}$, then $A^{(u,v)}(z;n)$ is bounded above for all $n\geq 0$. In particular, $$v+\frac{\log{2}}{u}-\lim_{n\to\infty}A^{(u,v)}(z;n)= O\bigg(\frac{1}{u^2v}\bigg).$$
 \end{theorem}

\begin{proof}
For brevity, we let $S(n):=S^{(u,v)}(z;n)$, $A(n):=A^{(u,v)}(z;n)$, and $\mathcal{T}(n):=\mathcal{T}^{(u,v)}(z;n)$.

For $n\geq 1$, every rational number in the set $\mathcal{T}(n)$ is either the left-child or right-child of a rational number in the set $\mathcal{T}(n-1)$. In particular, for every $y\in \mathcal{T}(n-1)$, there is a unique $x\in \mathcal{T}(n)$ that is \emph{the} right-child $y$. By definition, $x=y+v$. Likewise, there is a unique $z\in \mathcal{T}(n)$ that is \emph{the} left-child $y$, making $z=\frac{1}{u+\frac{1}{y}}$. It follows that
\begin{align}
S(n) &= S(n-1) + 2^{n-1}v + \sum_{y\in \mathcal{T}(n-1)}\frac{1}{u+\frac{1}{y}}.\label{ssum}
\end{align}
By dividing both sides of \eqref{ssum} by $2^n$, we immediately obtain the equality
\begin{align}
A(n) &= \frac{1}{2}A(n-1) + \frac{v}{2} + \frac{1}{2^n}\sum_{y\in \mathcal{T}(n-1)}\frac{1}{u+\frac{1}{y}}.\label{stildesum}
\end{align}
By induction on \eqref{stildesum}, we can express $A(n)$ as
\begin{align}
A(n) &= \frac{1}{2^n}A(0)+v\sum_{k=1}^n\frac{1}{2^k}+\frac{1}{2^n}\sum_{k=1}^n\sum_{y\in \mathcal{T}(n-k)}\frac{1}{u+\frac{1}{y}}\notag\\
&= \frac{z}{2^n}+v\bigg(1-\frac{1}{2^n}\bigg)+\frac{1}{2^n}\sum_{k=1}^n\sum_{y\in \mathcal{T}(n-k)}\frac{1}{u+\frac{1}{y}}\label{sfullformula}
\end{align}
Taking the limit as $n\to\infty$ of both sides of \eqref{sfullformula} shows that, to complete the proof, it is enough to prove that
\begin{align}
\lim_{n\to\infty}\frac{1}{2^n}\sum_{k=1}^n\sum_{y\in \mathcal{T}(n-k)}\frac{1}{u+\frac{1}{y}} &= \frac{\log{2}}{u}+O\bigg(\frac{1}{u^2v}\bigg).\label{usum}
\end{align}

Let $m=\lfloor n/2\rfloor$. We split the double sum in \eqref{usum} into two parts,
\begin{align}
\sum_{k=1}^n\sum_{y\in \mathcal{T}(n-k)}\frac{1}{u+\frac{1}{y}} &= \sum_{k=1}^{m}\sum_{y\in \mathcal{T}(n-k)}\frac{1}{u+\frac{1}{y}} + \sum_{k=m+1}^n\sum_{y\in \mathcal{T}(n-k)}\frac{1}{u+\frac{1}{y}}.\label{2sumsplit}
\end{align}
For $m<k\leq n$, we apply the following simple upper bound in \eqref{2sumsplit},
\begin{align*}
\sum_{y\in \mathcal{T}(n-k)}\frac{1}{u+\frac{1}{y}} &\leq \frac{2^{n-k}}{u}.
\end{align*}
It follows that
\begin{align}
\sum_{k=m+1}^n\sum_{y\in \mathcal{T}(n-k)}\frac{1}{u+\frac{1}{y}} &\leq \sum_{k=m+1}^n\frac{2^{n-k}}{u}\notag\\
&= \frac{1}{u}\sum_{i=0}^{n-(m+1)} 2^i\notag\\
&= \frac{2^{n-m}-1}{u}.\label{simpleupbound}
\end{align}
Since $m\to\infty$ as $n\to\infty$, if we apply \eqref{simpleupbound} to \eqref{2sumsplit}, then, by \eqref{usum}, we have reduced the problem to showing that
\begin{align}
\lim_{n\to\infty}\frac{1}{2^n}\sum_{k=1}^m\sum_{y\in \mathcal{T}(n-k)}\frac{1}{u+\frac{1}{y}} &= \frac{\log{2}}{u}+O\bigg(\frac{1}{u^2v}\bigg).\label{finalform}
\end{align}

Using the same reasoning on the sum $\displaystyle\sum_{y\in \mathcal{T}(n-k)}\frac{1}{u+\frac{1}{y}}$ that led to \eqref{ssum}, we see that, for $n-k>2$,
\begin{align}
\sum_{y\in \mathcal{T}(n-k)}\frac{1}{u+\frac{1}{y}} &= \sum_{y\in \mathcal{T}(n-(k+1))}\frac{1}{2u+\frac{1}{y}} + \sum_{y\in \mathcal{T}(n-(k+1))}\frac{1}{u+\frac{1}{v+y}}.\label{splitusum}
\end{align}
We convert the rightmost sum on the right-hand side of \eqref{splitusum} into a sum of geometric series,
\begin{align}
\sum_{y\in \mathcal{T}(n-(k+1))}\frac{1}{u+\frac{1}{v+y}} &= \frac{1}{u}\sum_{y\in \mathcal{T}(n-(k+1))}\frac{1}{1+\frac{1}{u(v+y)}}\notag\\
&= \frac{1}{u}\sum_{y\in \mathcal{T}(n-(k+1))}\sum_{j=0}^\infty\Bigg(\frac{-1}{u(v+y)}\Bigg)^j.\label{ugeomsum}
\end{align}
The justification for \eqref{ugeomsum} follows from the fact that $0<\frac{1}{u(v+y)}\leq \frac{1}{uv}\leq\frac{1}{2}$ for \emph{any} positive rational $y$. So
\begin{align}
\sum_{y\in \mathcal{T}(n-(k+1))}\frac{1}{u+\frac{1}{v+y}} &= \frac{1}{u}\sum_{y\in \mathcal{T}(n-(k+1))}\Bigg(1+O\bigg(\frac{1}{uv}\bigg)\Bigg)\notag\\
&= \frac{2^{n-(k+2)}}{u}\Bigg(1+O\bigg(\frac{1}{uv}\bigg)\Bigg)\label{firstestimate}
\end{align}
Combining \eqref{firstestimate} with \eqref{splitusum}, we see that
\begin{align*}
\sum_{y\in \mathcal{T}(n-k)}\frac{1}{u+\frac{1}{y}} &= \sum_{y\in \mathcal{T}(n-(k+1))}\frac{1}{2u+\frac{1}{y}} + \frac{2^{n-(k+2)}}{u}\Bigg(1+O\bigg(\frac{1}{uv}\bigg)\Bigg).
\end{align*}
We can now repeat all of the above steps starting from \eqref{splitusum} with the sum $$\sum_{y\in \mathcal{T}(n-(k+1))}\frac{1}{2u+\frac{1}{y}}.$$ Inductively, for any positive integer $j<n-k$, it follows that
\begin{align}
\sum_{y\in \mathcal{T}(n-k)}\frac{1}{u+\frac{1}{y}} &= \sum_{y\in \mathcal{T}(n-(k+j))}\frac{1}{(j+1)u+\frac{1}{y}} + \sum_{i=1}^j\frac{2^{n-(k+i+1)}}{iu}\Bigg(1+O\bigg(\frac{1}{uv}\bigg)\Bigg)\label{msum}
\end{align}
where the constant associated with the big-oh term is uniform for all of the sums.

Let $m'=\lfloor n/4\rfloor$. Then, from \eqref{msum}, for $1\leq k\leq m$,
\begin{align}
\sum_{y\in \mathcal{T}(n-k)}\frac{1}{u+\frac{1}{y}} &= \sum_{y\in \mathcal{T}(n-(k+m'))}\frac{1}{(m'+1)u+\frac{1}{y}} + \sum_{i=1}^{m'}\frac{2^{n-(k+i+1)}}{iu}\Bigg(1+O\bigg(\frac{1}{uv}\bigg)\Bigg)\notag\\
&= O\Bigg(\frac{2^{n-(k+m'+1)}}{(m'+1)u}\Bigg)+\sum_{i=1}^{m'}\frac{2^{n-(k+i+1)}}{iu}\Bigg(1+O\bigg(\frac{1}{uv}\bigg)\Bigg).\label{mprimebound}
\end{align}
(Note that for $n$ sufficiently large, since $k\leq m$, then $k+m'\leq 3n/4$, so $n-(k+m')\geq 1$. In particular, we can apply \eqref{msum} with $j=m'$.)

Using the Taylor series expansion of $\log(1-x)$ for $|x|<1$, we see that
\begin{align}
\sum_{i=1}^{m'}\frac{2^{n-(k+i+1)}}{iu} &= \frac{2^{n-(k+1)}}{u}\sum_{i=1}^{m'}\frac{1}{i2^i}\notag\\
&= \frac{2^{n-(k+1)}}{u}\bigg(\log{2}-\sum_{i>m'}\frac{1}{i2^i}\bigg).\label{logapprox}
\end{align}
Combining \eqref{mprimebound} and \eqref{logapprox} with the double sum from \eqref{finalform}, it follows that
\begin{align}
&\quad\dfrac{1}{2^{n-1}}\sum_{k=1}^m\sum_{y\in \mathcal{T}(n-k)}\frac{1}{u+\frac{1}{y}} \nonumber\\
&= \dfrac{1}{u}\displaystyle\sum_{k=1}^m\frac{1}{2^k}\bigg(\log{2}-\sum_{i>m'}\frac{1}{i2^i}\bigg)\Bigg(1+O\bigg(\frac{1}{uv}\bigg)\Bigg)+O\bigg(\frac{1}{(m'+1)u}\bigg)\label{lastform}
\end{align}
The result \eqref{finalform} now follows from taking the limit of \eqref{lastform} as $n\to\infty$.  
\end{proof}

Lemma~\ref{mono} and Theorem~\ref{mainthm} immediately give (A). To show (B), we give a crude estimate of the difference between two rational numbers based on their short continued fraction representations.

\begin{lemma}\label{bound}
Suppose that $\alpha,\beta\in\mathbb{Q}$ are distinct with $\alpha=[p_0,p_1,\dots,p_s]$ and $\beta=[q_0,q_1,\dots,q_r]$. Let $k$ be the largest index such that $p_k= q_k$. Then $$|\alpha-\beta|\leq\prod_{j=1}^{k}\frac{1}{p_j^2}.$$
\end{lemma}

\begin{proof}
We rewrite the continued fraction representations of $\alpha$ and $\beta$ as $$\alpha=[p_0,p_1,\dots,p_k,p_{k+1},\dots,p_s] \quad\text{and}\quad  \beta=[p_0,p_1,\dots,p_k,q_{k+1},\dots,q_r].$$ (Note that we cannot have $k=r=s$ and that if $k=r$ or $k=s$, the estimates below still apply.) Now, for $A_i=[p_i,\dots,p_s]$ and $B_i=[q_i,\dots,q_r]$ with $1\leq i\leq k+1$,
\begin{align*}
    |\alpha-\beta| &= \left|p_0+\dfrac{1}{p_1+A_1}-p_0-\dfrac{1}{p_1+B_1}\right|\\
    &= \left|\dfrac{1}{p_1+A_1}-\dfrac{1}{p_1+B_1}\right|\\
    &\leq  \left|p_1+\dfrac{1}{p_2+A_2}-p_1-\dfrac{1}{p_2+B_2}\right|\cdot\frac{1}{p_1^2}\\
    &\hspace{3cm}\vdots\\
    &\leq \left|\dfrac{1}{p_{k+1}+A_{k+1}}-\dfrac{1}{q_{k+1}+B_{k+1}}\right|\cdot\prod_{j=1}^{k}\frac{1}{p_j^2}\\
    &\leq \prod_{j=1}^{k}\frac{1}{p_j^2}. 
\end{align*}  
\end{proof}

In the case where the rationals from Lemma~\ref{bound} are vertices of possibly two different $(u,v)$-Calkin-Wilf trees, we get the following corollary.

\begin{corollary}\label{corbound}
With $\alpha$ and $\beta$ as in Lemma~\ref{bound} and, additionally, suppose that $\alpha$ and $\beta$ are vertices of possibly two different $(u,v)$-Calkin-Wilf trees, then
$$\alpha-\beta= O\Bigg(\frac{\max\{u,v\}}{2^k}\Bigg).$$
\end{corollary}

\begin{proof}
The corollary follows from the fact that if the two rationals $\alpha$ and $\beta$ are vertices on $(u,v)$-Calkin-Wilf trees, then $p_i$ is divisible by $v$ for even $i$ and divisible by $u$ for odd $i$ by Lemma~\ref{cfterms}. 
\end{proof}

Before we begin our proof of (B), we need one additional lemma.

\begin{lemma}\label{mcount}
Let $y=[q_0,q_1,\dots,q_r]$ with $q_r\neq 1$ when $y\neq 1$ and $r=0$ when $y=1$ and define $\ell(y)=r$. Let $f_z(n,m)=\#\{y\in\mathcal{T}^{(u,v)}(z;n):\ell(y)=m+\ell(z)\}$, then  for $m\geq 0$, 
\begin{align*}
    f_z(n,m) = \begin{cases}\displaystyle\binom{n+1}{m}& \text{ if } 2\nmid m \text{ and }z>1\\
    \displaystyle\binom{n+1}{m+1}& \text{ if } 2\nmid m \text{ and } z<1\\
    \displaystyle\binom{n}{m}& \text{ if } z=1\\
    0& \text{ otherwise.}\end{cases}
\end{align*}
\end{lemma}

\begin{proof}
The desired result can be shown to be true for $n<2$ by inspection. 

Assume that the statement is true for all $0\leq j\leq k$ for some $k\geq 2$ and let $y\in\mathcal{T}^{(u,v)}(z;k+1)$ be such that $\ell(y)=m+\ell(z)$. That is, we assume $y$ is a rational number counted by $f_z(k+1,m)$. There is a sequence of rational numbers $z_0=z,z_1,\dots,z_{k+1}=y$ such that $z_{i+1}$ is a descendant of $z_i$ for $0\leq i<k+1$. By Lemma~\ref{cfterms}, we see that $\ell(z_{i+1})-\ell(z_i)\in\{0,1,2\}.$ In fact, for $i\geq 1$, $\ell(z_{i+1})-\ell(z_i)=2$ if and only if $z_{i+1}$ is a left child of $z_i$ and $z_i$ is a right child of $z_{i-1}$, $\ell(z_1)-\ell(z_0)=2$ if and only if $z_1$ is a left child of $z_0$ with $z_0>1$, and $\ell(z_1)-\ell(z_0)=1$ if and only if $z_1$ is a left child of $z_0$ with $z_0=1$. 

We now consider the following three cases:

\noindent\underline{Case 1:} $z_2$ is a right child of $z_1$ and $z_1$ is a right child of $z_0.$

In this case we have that $y\in\mathcal{T}^{(u,v)}(z_2;k-1)$ with $\ell(y)=m+\ell(z_2).$

\noindent\underline{Case 2:} $z_2$ is a left child of $z_1$ and $z_1$ is a right child of $z_0.$

In this case we have that $y\in\mathcal{T}^{(u,v)}(z_2;k-1)$ with $\ell(y)=m-2+\ell(z_2).$

\noindent\underline{Case 3:} $z_1$ is a left child of $z_0.$

In this case we have that $y\in\mathcal{T}^{(u,v)}(z_1;k)$ with 
\begin{align*}
\ell(y)=\begin{cases}
m-2+\ell(z_1)& \text{ if $z_0>1$}\\
m+\ell(z_1)& \text{ if $z_0<1$}\\
m-1+\ell(z_1)& \text{ if $z_0=1$}.
\end{cases}
\end{align*}

It follows from the three cases above that,
\begin{equation}\label{recursion}
    f_z(k+1,m) = \begin{cases}
    f_{z'}(k-1,m)+f_{z''}(k-1,m-2)+f_{z'''}(k,m-2)& \text{ if $z_0>1$}\\
    f_{z'}(k-1,m)+f_{z''}(k-1,m-2)+f_{z'''}(k,m)& \text{ if $z_0<1$}\\
    f_{z'}(k-1,m)+f_{z''}(k-1,m-2)+f_{z'''}(k,m-1)& \text{ if $z_0=1$}.
    \end{cases}
\end{equation}
where $z'=z_0+2v>1$, $z''=\frac{1}{u+\frac{1}{v+z_0}}<1$, and $z'''=\frac{1}{u+\frac{1}{z_0}}<1.$

We will now make heavy use of the well-known binomial coefficient identity $\binom{n}{m}=\binom{n-1}{m}+\binom{n-1}{m-1}$ to complete the proof.

For $z_0>1$, the desired result is trivially true when $2\mid m$, so we assume otherwise. Therefore, by assumption
\begin{align*}
    f_z(k+1,m) &= \binom{k}{m}+\binom{k}{m-1}+\binom{k+1}{m-1}\\
    &= \binom{k+1}{m}+\binom{k+1}{m-1}\\
    &= \binom{k+2}{m}.
\end{align*}

Similarly, for $z_0<1$, the desired result is also trivially true when $2\mid m$, so we assume otherwise. Therefore, by assumption
\begin{align*}
    f_z(k+1,m) &= \binom{k}{m}+\binom{k}{m-1}+\binom{k+1}{m+1}\\
    &= \binom{k+1}{m}+\binom{k+1}{m+1}\\
    &= \binom{k+2}{m+1}.
\end{align*}

Finally, for $z_0=1$, by assumption, when $m$ is odd,
\begin{align*}
    f_z(k+1,m) &= \binom{k}{m}+\binom{k}{m-1}+0\\
    &= \binom{k}{m}+\binom{k}{m-1}\\
    &= \binom{k+1}{m}
\end{align*}
and when $m$ is even,
\begin{align*}
    f_z(k+1,m) &= 0+0+\binom{k+1}{m}\\
    &= \binom{k+1}{m}.
\end{align*}
Having exhausted all possibilities, we complete the proof by induction. 
\end{proof}

An application of the de Moivre–-Laplace limit theorem~\cite[p. 186]{F} shows that the number of continued fraction coefficients in depth $n$ elements is normally distributed with mean approximately $n/2$.

Corollary~\ref{corbound} and Lemma~\ref{mcount} can now be used to compare the difference between rationals in different $(u,v)$-Calkin-Wilf trees that are in the same position relative to the root, showing that the mean values of the rows for different trees are asymptotically the same.

\begin{proposition}\label{samelim}
For any $z_1,z_2\in[1/u,v]\cap\mathbb{Q}$, we have that $$A^{(u,v)}(z_1;n)-A^{(u,v)}(z_2;n)\to 0$$ as $n\to\infty.$
\end{proposition}

\begin{proof}
We begin by considering the case where $z_1=1/u$ and $z_2=v$. Let $y\in\mathcal{T}^{(u,v)}(v;n)$. Then by Lemma~\ref{cfterms} and Lemma~\ref{cfdepth}, $y$ has a continued fraction representation of the form $y=[\alpha_0v,\alpha_1u,\dots,\alpha_kv]$ with $\sum_{i=0}^k\alpha_i=n+1$. Consider the map $f:\mathcal{T}^{(u,v)}(v;n)\to\mathcal{T}^{(u,v)}(1/u;n)$ given by
\begin{align*}
    f(y)=\begin{cases}[\alpha_0v,\alpha_1u,\dots,(\alpha_{k-1}+1)u]& \text{ if $\alpha_k=1$}\\
    [\alpha_0v,\alpha_1u,\dots,(\alpha_k-1)v,u]& \text{ otherwise.}\end{cases}
\end{align*}
It is clear that $f$ represents a well-defined bijection. In particular, by Corollary~\ref{corbound} and Lemma~\ref{mcount},
\begin{align*}
    A^{(u,v)}\left(\frac{1}{u};n\right)-A^{(u,v)}(v;n) 
    &= \frac{1}{2^n}\sum_{y\in\mathcal{T}^{(u,v)}(v;n)}f(y)-y\\
    &= O\Bigg(\frac{\max\{u,v\}}{2^n}\Bigg(\sum_{y\in\mathcal{T}^{(u,v)}(v;n),a_k=1}\frac{1}{2^{k-1}}+\sum_{y\in\mathcal{T}^{(u,v)}(v;n),a_k>1}\frac{1}{2^k}\Bigg)\Bigg)\\
    &= O\Bigg(\frac{\max\{u,v\}}{2^n}\sum_{y\in\mathcal{T}^{(u,v)}(v;n)}\frac{1}{2^k}\Bigg)\\
    &= O\Bigg(\frac{\max\{u,v\}}{2^n}\sum_{k=0}^{n+1}\binom{n+1}{k}\frac{1}{2^k}\Bigg)\\
    &= O\Bigg(\max\{u,v\}\cdot \left(\frac{3}{4}\right)^n\Bigg),
\end{align*}
which goes to 0 as $n\to\infty.$

The cases $z_1=1/u$ and $z_2\in(1/u,1]\cap\mathbb{Q}$ and $z_1=v$ and $z_2\in[1,v)\cap\mathbb{Q}$ can be handled in a similar way. These three cases complete the proof of the proposition.
 
\end{proof}

Proposition~\ref{samelim} completes the proof of (B), giving the desired result.

\section{Acknowledgement}
The second author received support for this project provided by a PSC-CUNY Award, \#611157-00 49, jointly funded by The Professional Staff Congress and The City University of New York.


\begin{thebibliography}{99.}
\bibitem{A}
Alkauskas, G.: The moments of Minkowski question mark function: the dyadic period function. Glasg. Math. J. \textbf{52}(1), 41--64 (2010).
\bibitem{BM} 
Bates, B., Mansour T.:  The {$q$}-Calkin-Wilf tree. J. Combin. Theory Ser. A \textbf{118}, no. 3, 1143--1151  (2011).
\bibitem{CW}
Calkin, N., Wilf, H.S.:  Recounting the rationals. Amer. Math. Monthly \textbf{107}, no. 4, 360--363 (2000).
\bibitem{F}
Feller, W.: An Introduction to Probability Theory and its Applications, Vol. I, 3rd edition, John Wiley \& Sons Inc., 1968.
\bibitem{HMST1}
Han, S., Masuda, A.M., Singh, S.,  Thiel, J.: The (u,v)-Calkin-Wilf forest. Int. J. Number Theory \textbf{12}, no. 5, 1311--1328 (2016).
\bibitem{HMST2}
Han, S., Masuda, A.M., Singh, S., Thiel, J.: Orphans in forests of linear fractional transformations. Electron. J. Combin. \textbf{23}, no. 3, Paper 3.6, 24pp (2016).
\bibitem{MS} 
Mansour, T., Shattuck, M.:  Two further generalizations of the Calkin-Wilf tree, J. Comb. {\bf 2}, no. 4, 507--524 (2011).
\bibitem{N}
Nathanson, M.B.: A forest of linear fractional transformations, Int. J. Number Theory {\bf 11}, no. 4, 1275--1299  (2015).
\bibitem{Ne} 
Newman, M.: Recounting the rationals, continued, solution to problem 10906, Amer. Math. Monthly  \textbf{110}, 642--643  (2003).
\bibitem{R}
Reznick, B.: Regularity properties of the Stern enumeration of the rationals, J. Integer Seq. \textbf{11}, Article 08.4.1, 17pp (2008).

\bibitem{Ro}
Roberts, J.: Elementary Number Theory: A Problem Oriented Approach, MIT Press, 1977.

\end{thebibliography}
\end{document}